\documentclass[12pt]{article}

\usepackage{amssymb}
\usepackage{amsthm}
\usepackage{latexsym}
\usepackage{eufrak}

\title{Spinor class fields for sheaves of lattices}
\author{Luis Arenas-Carmona\thanks{Supported by Fondecyt,
proyecto No. 1085017.}
\\Universidad de Chile,
\\ Facultad de Ciencias,
\\Casilla 653, Santiago,Chile
\\E-mail: learenas@uchile.cl}

\begin{document}

\theoremstyle{plain}
\newtheorem{thm}{Theorem}
\newtheorem{prop}{Proposition}
\newtheorem{lemma}[prop]{Lemma}
\newtheorem{coro}{Corollary}[prop]
\newtheorem{corth}{Corollary}[prop]
\theoremstyle{definition}
\newtheorem{ex}{Example}
\newtheorem{dfn}[prop]{Definition}
\newtheorem{qu}{Open question}
\newtheorem{rmk}{Remark}
\newtheorem{nota}{Remark}

\newcommand\finitum{\mathbb{F}}
\newcommand\linitum{\mathbb{L}}
\newcommand\dif[1]{\mathrm{d}{#1}}
\newcommand\Top{\mathbb{T}}
\newcommand\Na{\mathcal{N}}
\newcommand\Ha{\mathfrak{H}}

\newcommand\hilbertfield{\mathbf{H}}
\newcommand\RK{R_K}
\newcommand\RSigma{R_\hilbertfield}
\newcommand\ord{\mathrm{ord}}

\newcommand\idealb{\mathcal{B}}
\newcommand\pipsi{\Pi_v\Psi_v}
\newcommand\alge{\mathfrak{A}}
\newcommand\ad{\mathbb A}
\newcommand\rank{\textnormal{rank }}
\newcommand\Aut{\textnormal{Aut }}
\newcommand\Hom{\textnormal{Hom}}
\newcommand\modstar{\textnormal{mod}^*\ }
\newcommand\tr{\textnormal{ tr}}
\newcommand\la{\Lambda}
\newcommand\bark{{\bar{k}}}
\newcommand\pe{\mathfrak{P}}
\newcommand\uno{\{1\}}
\newcommand\quadc{k^*/(k^*)^2}
\newcommand\oink{\mathcal O}
\newcommand\oinkkR{\oink_{k,R}}
\newcommand\oinkKR{\oink_{K,\RenK}}
\newcommand\oinkFR{\oink_{F,R_F}}
\newcommand\A{\mathfrak A}
\newcommand\Q{\mathbb Q}
\newcommand\imap{\mathcal I}
\newcommand\reali{\mathbb R}
\newcommand\compleji{\mathbb C}
\newcommand\enteri{\mathbb Z}
\newcommand\barkv{\bar{k}_v}
\newcommand\gal{\mathcal G}
\newcommand\hal{\mathcal H}
\newcommand\galbark{{\mathcal G_{\bark/k}}}
\newcommand\galbarkvkv{{\mathcal G_{\barkv/k_v}}}
\newcommand\galbarkk{\galbark}
\newcommand\galkk{{\mathcal G_{K/k}}}
\newcommand\galkvkv{{\mathcal G_{K_w/k_v}}}
\newcommand\galv{{\mathcal G_w}}

\newcommand\cero{\{0\}}
\newcommand\coind[3]{\textnormal{Coind}_{#1}^{#2}(#3)}
\newcommand\ind[3]{\textnormal{Ind}_{#1}^{#2}(#3)}
\newcommand\vaa{\longrightarrow}
\newcommand\baja{\downarrow}
\newcommand\ort{\perp}
\newcommand\grupi{\mathbb G}
\newcommand\Da{\mathfrak{D}}
\newcommand\Ba{\mathfrak{B}}
\newcommand\DaKwp{\Da_{K_{\wpK}}}
\newcommand\DaK{\Da_K}
\newcommand\DaKun{\DaK^1}

\newcommand\vekv{V_{k_v}}
\newcommand\idealc{\mathcal{C}}
\newcommand\ideala{\mathcal A}
\newcommand\ideale{\mathcal E}
\newcommand\idealf{\mathcal F}
\newcommand\idealp{\wp}
\newcommand\idealP{\mathcal P}
\newcommand\quat{\mathcal D}
\newcommand\D{\quat}
\newcommand\orden{\mathfrak{D}}
\newcommand\rscal{\mathcal{R}_{L/k}}

\newcommand\bigperpp{\mathop{\textnormal{\Huge $\perp$}}}
\newcommand\bimatrix[4]{\left(\begin{array}{cc}#1&#2\\#3&#4\end{array}
\right)}
\newcommand\tmatrix[9]{\left(\begin{array}{ccc}#1&#2&#3\\#4&#5&#6\\
#7&#8&#9\end{array}\right)}

\newcommand\spinlomegag{Spin_{2n}(g)}

\newcommand\spinomega{Spin_{n}(Q)}
\newcommand\hspinomega{Spin_{n}(\quat,h)}

\newcommand\tensor{\bigotimes_{\oink_k}}
\newcommand\classa{[a_\sigma]_\sigma}
\newcommand\classu{[u_\sigma]_{\sigma\in\gal}}
\newcommand\defaction{f^{-1}(a^{-1}_\sigma(f(w))^\sigma)}
\newcommand\matrici{\mathbb M}
\newcommand\Zmod[1]{\enteri/{#1}\enteri}

\maketitle

\begin{abstract}
  We extend the theory of spinor class field and representation fields
to any linear algebraic group over a global function field
satisfying some technical conditions that ensure the existence of
a suitable spinor norm. This is the analog of a result given by
the author in the number field case. Spinor class fields can still
be defined for lattices defined over a projective curve in a
sheaf-theoretical context. Spinor genus is a rather weak invariant
in this context, by it can be used to study the behavior of the
genus at affine subsets. Examples are provided.
\end{abstract}

\section{Introduction}\label{intro} Let $n$ be a positive integer
and let $k$ be a number field with ring of integers $\oink_k$. Every conjugacy class of maximal
orders in the matrix algebra $\matrici_n(k)$ has a representative
of the form $\Da_n(I)=(I^{\delta_{1,i}-\delta_{1,j}})_{i,j}$,
where $(I_{i,j})_{i,j}$ is the lattice of matrices
$(a_{i,j})_{i,j}$ satisfying $a_{i,j}\in I_{i,j}$ for ideals $I_{i,j}$ (\cite{Chev}, p.
18), i.e.,
\[
\Da_2(I)=\bimatrix{\oink_k}I{I^{-1}}{\oink_k},
\Da_3(I)=\tmatrix{\oink_k}II{I^{-1}}{\oink_k}{\oink_k}{I^{-1}}{\oink_k}{\oink_k},\dots
\]
Two such orders $\Da_n(I)$ and $\Da_n(I')$ are conjugate if and
only if $J^nI'I^{-1}$ is principal for some fractional ideal $J$
of $k$ (\cite{Chev}, p. 23). It follows that the set of conjugacy
classes of maximal orders in $\matrici_n(k)$ is in correspondence
with the quotient $\mathfrak{g}/\mathfrak{g}^n$ of the class group
$\mathfrak{g}$ of $k$. In \cite{abelianos} we extended this theory
to the case where $k$ is a global function field and $\oink_k$ is
an arbitrary Dedekind domain whose field of quotients is $k$. We
can even consider the projective case if we replace the ring
$\oink_k$ by the structure sheaf $\oink_X$ of a smooth projective
curve $X$ whose field of rational functions is $k$, and we
interpret maximal orders and ideals in the preceding statements in
a sheaf-theoretical setting. In that context, it can only be
proved that every spinor genera of maximal orders can have a
representative of the form $\Da_n(I)$. In fact, the existence of
non-split maximal orders in $\frac1n$-th of all spinor genera is a
consequence of the Theory of representation by spinor genera
(\S\ref{s4}). In this article we extend this theory to any group
having an apropiate spinor norm, as it was done for the number
field case in \cite{spinor}.

 Let $X$ be a smooth projective curve over a finite field $\finitum$.
 Let $K(X)$ be the field of rational functions on $X$, and
let $G$ be an algebraic subgroup of $\mathrm{Aut}_{K(X)}{V}$, for
some $K(X)$-vector space $V$. Let $C$ be an arbitrary non-empty
Zariski-open subset of $X$. We allow the case $C=X$. Let $\Lambda$
and $M$ be $C$-lattices on $V$. We say that $\Lambda$ and $M$ are:
\begin{enumerate} \item in the same $G$-class if $g(\Lambda)=M$
for some $g\in G$, \item in the same $G$-genus if for every point
$\wp\in C$ there exists an element $g_\wp\in G_\wp$
such that $g_\wp(\Lambda_\wp)=M_\wp$, where
$\Lambda_\wp$ and $M_\wp$ denote the completions  at
$\wp$, \item in the same $G$-spinor genus if there exists a
lattice $P$ satisfying the following conditions:
\begin{enumerate}
\item There exists an element $h\in G$ such that $h(\Lambda)=P$,
\item  For every point $\wp\in C$ there exists an element
$g_\wp\in G_\wp$ with trivial spinor norm (\S\ref{s1b}) such that
$g_\wp (P_\wp)=M_\wp$.
\end{enumerate}
\end{enumerate}
Spinor genera and classes coincide whenever $G$ is non-compact at
some infinite point of $C$ if any (Proposition \ref{p31b}). In
\S\ref{s2} we prove the following result:
\begin{thm}\label{main}
Let $G$ be any semi-simple linear algebraic group $G$, defined
over the field of functions $K=K(X)$ of a smooth projective curve
$X$. Assume $G$ satisfies the technical conditions \textbf{SN} and
\textbf{RU} in \S\ref{s1b} below. Then the set of $G$-spinor
genera in the $G$-genus of any $C$-lattice $\Lambda$ is a
principal homogeneous space over the Galois group $\gal$ of an
Abelian extension $\Sigma^C_\Lambda/K(X)$ called the spinor class
field of the lattice $\Lambda$. This extension splits completely
at every infinite point of $C$ if any. If $D\subseteq C$ is open
and $M$ is the restriction of $\lambda$ to $D$ as a sheaf (or
equivalenlty the $\oink_X(D)$-lattice generated by $\Lambda$),
then $\Sigma^C_\Lambda/K(X)$ is the maximal subextension of
$\Sigma^C_\Lambda/K(X)$ splitting at every place in $C\backslash
D$.
\end{thm}

 In other words, we have a natural action of
$\gal$ on the set $\Omega$ of spinor genera in the genus, and for
every pair of lattices $(M,N)$ there exists a unique element
$\rho(M,N)\in\gal$ taking the spinor genus of $M$ to the spinor
genus of $N$. The map $\rho$ satisfies the relation
$\rho(M,P)=\rho(M,N)\rho(N,P)$ for any three lattices $M$, $N$,
and $P$ in the genus. The class field $\Sigma^C_\Lambda$ depends
only on the genus of the lattice $\Lambda$.

\subparagraph{Example $A$.} Assume $q$ is odd. Let
$\{C_i\}_{i=1}^m$ be an affine cover of an irreducible smooth
projective curve $X$ over $\finitum_q$. For $i=1,\dots,m$, let
$\Lambda_i$ be a $C_i$-lattice in a fixed regular quadratic or
skew-hermitian $K(X)$-space $W$. Assume that the completions
$\Lambda_{i,\wp}$ and $\Lambda_{j,\wp}$ coincide whenever $\wp\in
C_i\cap C_j$. Then there exists a class field $\Sigma$ such that,
for every affine subset $D$ of $X$, if $M$ is the $D$-lattice
satisfying $M_\wp=\Lambda_{i,\wp}$ for $\wp\in C_i\cap D$, then
the spinor class field $\Sigma^D_M$ is the maximal subfield of
$\Sigma$ splitting completely at the infinite places of $D$. We
simply define $\Sigma$ as the spinor class field of the lattice
obtained by pasting together the lattices $\Lambda_i$ as sheaves.

It is apparent that
$\Sigma\supseteq\Sigma^{C_1}_{\Lambda_1}\cdots\Sigma^{C_m}_{\Lambda_m}$,
but equality does not need to hold. For example, assume
$X=\mathbb{P}^1(\finitum_q)=C_1\cup C_2$ with either $C_i$ affine.
Let $\Lambda$ be a free $X$-lattice with a basis $v_1,\dots,v_n$
for some $n\geq 3$, and let $\Lambda_i$ be the restriction of
$\Lambda$ to $C_i$, for $i=1,2$.  Let $Q$ be the quadratic form
defined by
$$Q\left(\sum_{i=1}^n g_iv_i\right)=\sum_{i=1}^{n-2} g_i^2+g_{n-1}g_n.$$
Then every spinor genus has a representative of the form
$$L(I)=\bigperpp_{i=1}^{n-2}\oink_Cv_i\perp(Iv_{n-1}\oplus
I^{-1}v_n),$$ for a suitable ideal $I$.This shows that
$\Sigma^{C_i}_{\Lambda_i}=K(X)$ for either affine set $C_i$, but
we claim that $\Sigma$ must contain a quadratic extension. The
claim follows by considering an open subset $D$ isomorphic to the
rational affine curve with equation $y=P(x,y)$ where $P$ is an
irreducible quadratic form on $\finitum_q[x,y]$. It suffices to
prove that the class group of the ring $\oink_X(D)$ has even
order. In fact, the ideal $I=(x,y)$ is not principal, but
$I^2=(y)$, as an easy computation shows.

 In fact, the spinor genus
of the lattice $L(I)$ can be computed from the divisor class
defining the ideal $I$. We show in \S \ref{s4} that not every
class of lattices in the genus of $L(\oink_X)$ has a
representative of this type. We do this by extending the theory of
representation fields to the sheaf setting.  When strong
approximation fails, as is always the case for $X$-lattices,
representation fields give only information on the number of
spinor genera representing a given lattice. Just as in the number
field case, representation fields might fail to exists, but they
do exist in many important families of examples.

\section{Spinor norms}\label{s1b}

To fix ideas, let $X$ be the irreducible smooth projective  algebraic curve
defined over a finite field $\finitum=\finitum_q$,
 and let $K(X)$ be its field of rational functions.
 We say that a semi-simple linear algebraic group
$G\subseteq\mathrm{GL}\Big(n,K(X)\Big)$, defined over $K(X)$, satisfies
condition \textbf{SN} if:
\begin{enumerate}
 \item\emph{The extension of the universal cover
$\phi:\tilde{G}\rightarrow G$ of $G$ to any separably closed field
$E$ containing $K(X)$ is surjective.}
 \item\emph{ For almost all points $\wp\in X$, any integral
 element $g$ of the completion $G_\wp$ has a pre-image in $\widetilde{G}_Z$
 for some unramified extension $Z/K_\wp$.}\end{enumerate}

 Let $F=\ker\phi$ be the fundamental group of a semi-simple group $G$
satisfying condition \textbf{SN}. The cohomology function
$\theta:G\rightarrow H^1\Big(K(X),F_E\Big)$ arising from the short exact
sequence $F_E\hookrightarrow\tilde{G}_E\twoheadrightarrow G_E$,
where $E$ is the separable closure of $K(X)$, is called the spinor
norm. It is also defined at any field containing $K(X)$. The spinor
norm on a completion $K_\wp$ of $K(X)$ is denoted $\theta_\wp$.
There exists also an adelic version of the spinor norm. It is the
map
$$\Theta:
G_\ad\rightarrow\prod_{\wp\in X}H(K_\wp,F),\qquad
\Theta(g)=\Big(\theta_\wp(g_\wp)\Big)_\wp,$$ where
$G_\ad\subseteq\prod_\wp G_\wp$ is the adelization of the group
$G$ \cite{Pla}.

\begin{lemma}\label{surj}
For any semi-simple group $G$ satisfying condition \textbf{SN} the
spinor norm is surjective over $K(X)$ and over any localization
$K_\wp$.
\end{lemma}

\begin{proof} Recall that $H^1(K,\widetilde{G})=\{1\}$ for both,
the global field $K=K(X)$ and the local field $K=K_\wp$
\cite{serre95}. Now the result follows by applying cohomology to
the short exact sequence
$F_E\hookrightarrow\widetilde{G}_E\twoheadrightarrow G_E$, where
$E$ is the separable closure of $K(X)$ or $K_\wp$.\end{proof}

Let $E$ be the separable closure of $K(X)$. In what follows, we say
that $G$ satisfies condition \textbf{RU} if it satisfies the
following conditions:
\begin{enumerate}
\item\emph{ The fundamental group $F_E$ of $G_E$ is cyclic, and
its order $n$ is not divisible by the characteristic of
$\finitum$. In particular, it is isomorphic to the group $\mu_n$
of $n$-roots of unity in $E$.} \item\emph{ There exists an
isomorphism between $F_E$ and $\mu_n$ commuting with the natural
action of the Galois group $\mathrm{Gal}\Big(E/K(X)\Big)$ on
either group.}\end{enumerate} Condition \textbf{RU} implies that
$H^1\Big(K(X),F_E\Big)\cong K(X)^{*n}/K(X)^*$. Note that if
$E(\wp)$ is the separable closure of $K_\wp$, then
$F_{E(\wp)}=F_E$, since $K_\wp$ contains no inseparable extensions
of $K(X)$ (\cite{weil}, \S VIII.6.). It follows that
$H^1(K_\wp,F_E)\cong K_\wp^{*n}/K_\wp^*$ for any $\wp\in X$.

\begin{lemma}\label{gr}
For any semi-simple group $G$ satisfying conditions \textbf{SN}
and \textbf{RU} the image of the adelic spinor norm $\Theta$ is
contained in $J_X/J_X^n$, where $J_X$ is the idele group of $K(X)$.
\end{lemma}
\begin{proof} Observe that for any extension $E/K_\wp$
such that a given element $g\in G_{K_\wp}=G_\wp$ is in the image of the
cover $\phi:\widetilde{G}_E\rightarrow G_E$, the spinor norm
$\theta_\wp(g)$ can be computed from the short exact sequence
$F_E\hookrightarrow\widetilde{G}_E\twoheadrightarrow\phi(\widetilde{G}_E)$,
whence it lies in $H^1(E/K_\wp,F_E)\cong(E^{*n}\cap
K_\wp^*)/K_\wp^{*n}$. The latter isomorphism is just the
coboundary map of the short exact sequence $F_E\hookrightarrow E^*
\twoheadrightarrow E^{*n}$.

By condition \textbf{SN} every
$g\in G_{\oink_\wp}$ is the image of some $h\in \widetilde{G}_E$
for some unramified extension $E/K_\wp$, for almost all points $\wp\in X$. This implies
$$\theta_\wp(g)\in (E^{*n}\cap K_\wp^*)/K_\wp^{*n}\subseteq
\oink_\wp^*K_\wp^{*n}/K_\wp^{*n},$$ whence
$\theta_\wp(G_{\oink_\wp})\subseteq
\oink_\wp^*K_\wp^{*n}/K_\wp^{*n}$ for almost all $\wp$.\end{proof}

The automorphism groups of the structures mentioned in the
introduction indeed satisfy these conditions.

\begin{lemma}
Orthogonal groups of regular quadratic forms and unitary groups of
regular quaternionic skew-hermitian forms satisfy conditions
\textbf{SN} and \textbf{RU} if the characteristic of the base
field $\finitum$ is not $2$.
\end{lemma}

\begin{proof} The universal cover of the Orthogonal group of a
quadratic space $(V,q)$ over a field $K$ whose characteristic is
not $2$ is the spin group $\mathrm{Spin}(q)$. It is defined as the
set of elements $u$ in the Even Clifford Algebra $C^+(q)$
\cite{Om} satisfying $u\overline{u}=1$ and $uVu^{-1}=V$. Condition
\textbf{RU} follows since any element $u$ in the spin group
satisfying $uvu^{-1}=v$ for any $v\in V$ is in the base field
$K(X)$ (\cite{Om}, \S54:4), whence $u\overline{u}=1$ implies
$u=\pm1$.

If the field $E\supseteq K$ is separably closed, then every
product of two symmetries $\tau_v\tau_w$ is the image of
$\frac{vw}{\sqrt{q(v)q(w)}}\in \mathrm{Spin}(q)_E$. This elements
generate the special orthogonal group \cite{Om}. Furthermore, if
$K=K_\wp$ is a local field and $q$ is a unimodular integral
quadratic form at $\wp$, the integral orthogonal group is
generated by products of $2$ reflections $\tau_v\tau_w$ where
$q(v)$ and $q(w)$ are units (\cite{Om}, \S92:4), whence
$\frac{vw}{\sqrt{q(v)q(w)}}$ is defined over an unramified
extension.

For unitary groups of quaternionic skew-hermitian forms the proof
follows from the previous case, since
 any quaternion algebra splits on some separable quadratic extension of
 $K(X)$ (\cite{Reiner}, Thm. 7.15), which ramifies at only finitely many
 places (Theorem 1 in \S VIII.4 of \cite{weil}), and the unitary group of a skew-hermitian form on a split
 quaternion algebra is isomorphic to an orthogonal group (\cite{german}, Lemma 3).
\end{proof}

A similar result can be proved for the automorphism group of a
central simple algebra of dimension $n^2$ when $n$ is not
divisible by the characteristic of $\finitum$. However, we have a
stronger result:

\begin{lemma}\label{l24}
If $G$ is the automorphism group of a central simple algebra
$\alge$, the reduced norm map $\Theta=N:G_\ad\rightarrow
J_X/J_X^n$ satisfies the conclusions of Lemma \ref{gr} and Lemma
\ref{surj} regardless of the characteristic.
\end{lemma}

\begin{proof}
 See (\cite{weil}, \S X.2, Prop 6) and (\cite{weil},
\S XI.3, Prop 3) for Lemma \ref{surj}.

Passing to a separable extension if needed we assume that $\alge$
is isomorphic to a matrix algebra $\matrici_t\Big(K(X)\Big)$.
Restricting to a smaller set of points $\wp$ if needed, we assume
that the isomorphism maps the standard basis of the matrix algebra
to a basis of the lattice of integral elements in $\alge$. Then
the integral elements of $G$ are just the automorphisms of
$\matrici_t(K_\wp)$ fixing $\matrici_t(\oink_\wp)$, i.e.
$PGL(n,\oink_\wp)$. Any element $g\in PGL(n,\oink_\wp)$ has
reduced norm in $\oink_\wp^*K_\wp^{*n}$ and the conclusion of
Lemma \ref{gr} follows.
\end{proof}

We define the spinor norm for the automorphism group of an algebra
$\alge$ as the reduced norm as above. When
$\mathrm{char}(\finitum)$ divides $n$, the map
$SL_1(\alge_E)\rightarrow\mathrm{Aut}_{E}\alge_E$ fails to be
surjective, so we cannot interpret the spinor norm as a
co-boundary. However, the explicit construction of $\Theta$ is not
used in the remaining of this work.

\section{$X$-lattices}\label{s2}

In this section we recall the properties of $X$-lattices that are
used in the sequel. Let $X$ be a smooth irreducible projective curve
over a finite field $\finitum$.  Let $K=K(X)$ be the field of rational functions
on $X$, and for every place $\wp\in X$ we let $K_\wp$ be the
completion at $\wp$ of $K$. We let $\oink_\wp$ be the ring of
integers at $\wp$, i.e., the completion of the ring of rational
functions defined at $\wp$. Let $V$ be a vector space over $K$.
 A coherent system of lattices in $X$ is a family
$\{\Lambda_\wp\}_{\wp\in X}$ satisfying the following conditions
(\cite{weil}, Ch. VI, p.97):
\begin{enumerate}
\item Every  $\Lambda_\wp$ is a $\oink_\wp$-lattice in $V_\wp$.
\item There exists an affine set $C\subset X$ and a lattice $L$
over the ring $\oink_X(C)$, of rational functions defined
everywhere in $C$, such that $L_\wp=\Lambda_\wp$ for every $\wp\in
C$.
\end{enumerate}
Let $V_\ad$ be the adelization of the space $V$, and let us
identify $V$ with a discrete subgroup of $V_\ad$ as in
\cite{weil}. Then for any coherent system
$\tilde{\Lambda}=\{\Lambda_\wp\}_{\wp\in X}$, the product
$\Lambda_\ad=\prod_{\wp\in X}\Lambda_\wp$ is an open and compact
subgroup of $V_\ad$, and every open and compact
$\oink_\ad$-submodule of $V_\ad$ arises in this way (\cite{weil},
\S VI, Prop 1). For every affine subset $C$ of $X$ we define
$\Lambda_{\ad,C}=\prod_{\wp\in X\backslash C}V_\wp\times
\prod_{\wp\in C} \Lambda_\wp$. Then
$\Lambda(C)=\Lambda_{\ad,C}\cap V$ defines a sheaf $\Lambda$ on
$X$. We  call a sheaf of this type an $X$-lattice. Equivalently,
an $X$-lattice is a locally free sub-sheaf of $V$, where $V$ is
identified with the corresponding constant sheaf. Thus defined,
$X$-lattices share some of the properties of usual lattices,
namely:
\begin{itemize}
\item An $X$-lattice $\Lambda$ is completely determined by the coherent system
$\{\Lambda_\wp\}_\wp$.
\item  A coherent system can be modified at a finite number of places to define a
new $X$-lattice. In particular, $X$-lattices can be defined by gluing together
lattices defined over an affine cover.
 \item The adelization $\mathrm{GL}_\ad(V)$ of the general linear group
$\mathrm{GL}(V)$ of $V$ acts on the set of lattices by acting on
the family of compact and open $\oink_\ad$-submodules of $V_\ad$.
\item If $V=\alge$ is an algebra, an $X$-lattice $\Da$ is an order
(i.e., a sheaf of orders) if and only if every completion is an
order. The same holds for maximal orders.
\end{itemize}
For the proofs see \cite{Om} or \cite{Pla}\footnote{Although
\cite{Pla} assumes characteristic $0$ throughout, this hypotheses
is not used for the results quoted here.}.
For any linear algebraic group acting on the space $V$, we have an induced
action of the adelic group $G_\ad$ on the set of
$X$-lattices in $V$. Two $X$-lattices are in the same
$G$-genus if they are in the same orbit under this action. Similarly,
classes are characterized as $G_{K(V)}$-orbits and
spinor genera as $G_{K(V)}\mathrm{ker}(\Theta)$-orbits.
 It follows from our main theorem that there exist only a finite
number of spinor genera in a genus for any group $G$ satisfying
\textbf{SN} and \textbf{RU}. On the contrary, the number of classes
in a genus is frequently infinite, as in the examples we show in \S\ref{s4}.

 Note that $X$-lattices
are locally free sheaves over the structure sheaf of $X$, and
therefore they are associated to vector bundles \cite{atiyah}. The
assumption made here that an $X$-lattice is contained in the
constant sheaf $V$ is not restrictive since for any locally free
sheaf $\Lambda$ on $X$ the sheaf $V=\Lambda\otimes_{\oink_X}K$ is
constant and a finite vector space over $K$. Furthermore, any
isomorphism between two $X$-lattice in a space $V$ can be extended
to a linear map on $V$, whence next result follows:

\begin{prop}
The set of isomorphism classes of vector bundles of rank $n$ over a curve $X$
defined over a finite field is in correspondence with the set of
double cosets
$\mathrm{GL}(K,n)\backslash\mathrm{GL}(\ad_K,n)/\mathrm{GL}(\oink_\ad,n)$.\qed
\end{prop}

It is not true in general that this set of double cosets is finite
or that their elements are parameterized by their images under the
reduced norm as it is the case for lattices over affine subsets.
This follows easily from the classification results for vector
bundles over arbitrary fields. See for example \cite{atiyah}.

An $X$-lattice $\Lambda$ in a space $V$ is completely decomposable
if $\Lambda=\bigoplus_i J_iv_i$, where $\{v_1,\dots,v_n\}$ is a
basis of the space $V$ and $J_1,\dots,J_n$ are $X$-lattices in
$K(X)$. Note that every such lattice has the form $J_i=\mathfrak{L}^{B_i}$, where
$$\mathfrak{L}^{B}(C)=\left\{f\in K(X)\Big| \mathrm{div}(f)|_C\geq -B|_C\right\},$$
 for some divisor $B$ on $X$.
Not every $X$-lattice is
completely decomposable, as follows from the corresponding result
for vector bundles \cite{atiyah}. In next section we need the following result:

\begin{lemma}
There is a correspondence between conjugacy classes of maximal
$X$-orders in $\matrici_n(K)$ and isomorphism classes of $n$-dimensional vector
bundles over $X$ up to multiplication by invertible bundles.
\end{lemma}

\begin{proof}
Since all maximal orders are locally conjugate at all places, any
maximal $X$-order on $\alge$ has the form $b\Da_0b^{-1}$ where
$b\in\alge_{\ad}$ is a matrix with adelic coefficients and
$\Da_0\cong\matrici_n(\oink_X)$ is the sheave of matrices with
regular coefficients. We know that the adelization $\Da_{0\ad}$ is
the ring of all adelic matrices $c$ satisfying
$c(\oink_X^n)=\oink_X^n$. It follows that $b\Da_0b^{-1}$ is
the ring of all adelic matrices $c$ satisfying $c\Lambda=\Lambda$
where $\Lambda=b\Lambda_0=b(\oink_X^n)$. Since the stabilizer of a
local order $\Da_\wp$ is $\Da_\wp^*K_\wp^*$, it follows that two
$X$-lattices $\Lambda_1$ and $\Lambda_2$ corresponds to the same
maximal order, if and only if $\Lambda_1=d\Lambda_2$ for some
 $d$ in the group $J_X$ of ideles on $X$. The result follows since the idele $d$ generates
 the invertible bundle $\mathfrak{L}^{-\mathrm{div}(d)}$.
\end{proof}

Note that split orders correspond to completely decomposable bundles. In particular, not
all maximal orders are split.

\subparagraph{Proof of Theorem \ref{main}} The set of spinor
genera in a genus is in one to one correspondence with the Abelian
group
\begin{equation}\label{abln}G_\ad/G^{\Lambda}_{\ad}G_{K(V)}\ker(\Theta).
\end{equation} If the group
$G$ satisfies condition \textbf{RU}, then
$H(K_\wp,F)=K_\wp^*/K_\wp^{*n}$. By Lemma \ref{gr}, the image of
$\Theta$ is contained in $J_X/J_X^n$, where $J_X$ is the idele
group of $K(X)$. Since the spinor norm is surjective in both
$K(X)$ and $K_\wp$ (Lemma \ref{surj}), the group (\ref{abln}) is
isomorphic to $J_X/K(X)^*H(\Lambda)$, where $H(\Lambda)$ is the
pre-image in $J_X$ of the group $\Theta(G_\ad^\Lambda)\subseteq
J_X/J_X^n$. We let $\Sigma^C_\Lambda$ be the class field
associated to the open subgroup $K(X)^*H(\Lambda)$  (\cite{weil},
\S XIII.9).  The set of $G$-spinor genera in the $G$-genus of
$\Lambda$ is a principal homogeneous space, via Artin Map, for the
group $\gal=\mathrm{Gal}\Big(\Sigma^C_\Lambda/K(X)\Big)$. The
element of $\gal$ sending the spinor genus of an $X$-lattice $M$
to the spinor genus of a second $X$-lattice $M'$ is defined by
$\rho(M,M')=[a,\Sigma_\Lambda/K(X)]$, where
$x\mapsto[x,\Sigma_\Lambda/K(X)]$ denotes the Artin map, and $a$
is any element of $J_X$ satisfying $\Theta(g)=aJ_X^n$ for some
$g\in G_\ad$ such that $M'=g(M)$. The las statement follows from
the identity
$$H(L)=H(\Lambda)\left(\prod_{\wp\in C\backslash D}\theta_\wp(G_\wp)\right)=
H(\Lambda)\left(\prod_{\wp\in C\backslash D}K_\wp^*\right),$$
which follows from condition \textbf{SN} and the surjectivity of
$\theta_\wp$.
\qed

\paragraph{Example A (continued).}Let $\Lambda$ be the free
$X$-lattice with basis $\{v_i\}_{i=1}^n$ in the quadratic space
described in the introduction. Then
$\Sigma_\Lambda=\mathbb{L}(t)$, where $\mathbb{L}$ is the only
quadratic extension of $\finitum$, by a straightforward local
computation. To find a representative in every spinor genus we
observe that the adelic orthogonal element $g=a(\lambda)$ defined by
$g_\wp(v_i)=v_i$ for $i=1,\dots,n-2$, $g_\wp(v_{n-1})=\lambda_\wp
v_{n-1}$, and $g_\wp(v_n)=\lambda_\wp^{-1} v_n$ has spinor norm
$\lambda=(\lambda_\wp)_\wp$. Bow take $u\in J_X$,
and let $B=\mathrm{div}(u)$ be the corresponding divisor.
To find a representative $L_u$ of the spinor genus corresponding to the class of
 $u$ in $J_X/J_X^2H(\Lambda)$ we set $\lambda=u$ above, whence
$$L_u=a(u)\Lambda=\bigperpp_{i=1}^{n-2}\oink_Xv_i\perp(\mathfrak{L}^Bv_{n-1}\oplus
\mathfrak{L}^{-B}v_n).$$ Note that $L_u=L(B)$ depends only on the divisor $B$ of $u$.

\begin{prop}\label{p31b}
Assume that $C$ is a (necessarily proper) open set such that
$G_{X\backslash C}$ is non-compact. Then two $C$-lattices $N$ and
$M$ in the same $G$-genus are in the same $G$-class if and only if
$\rho(N,M)$ is the trivial element in
$\mathrm{Gal}\Big(\Sigma^C_N/K(X)\Big)$.
\end{prop}

\begin{proof}
By the strong approximation theorem over function fields
\cite{strong}, the universal cover $\widetilde{G}$ of $G$ has the
strong approximation property with respect to the set
$S=X\backslash C$. Assume that $\rho(N,M)$ is trivial. Then $N$
and $M$ are in the same spinor genus, i.e., there exist $g\in G$
and $h\in\ker\Theta$ such that $gh(N)=M$. Then any pre-image
$\tilde{h}$ of $h$ can be arbitrarily approximated by an element
$\tilde{f}$ in $\widetilde{G}$ whose image $f\in G$ approximates
$h$. Since lattice stabilizers in $G_\ad$ are open, the result
follows.
\end{proof}

Note that in example \textbf{A}, the group $G$ has
strong aproximation with respect two every non-empty finite subset
of $X$. It follows that two lattices are in the same spinor genus
if and only if they are conjugate over every affine subset of $X$.

\begin{coro}
Let $\Da$ and $\Da'$ be two maximal $C$-orders in the central
simple algebra $\alge$ that is not totally ramified at one or more
infinite places of $C$. Then $\Da$ and $\Da'$ are conjugate if and
only if $\rho(\Da,\Da')=\mathrm{id}$.
\end{coro}

\begin{coro}
Assume $\mathrm{char}(\finitum)\neq2$.
Let $N$ and $M$ be two $C$-lattices in the quadratic space $W$ of
dimension at least $3$ (or a skew-hermitian space of rank at least
2) that belong to the same genus. Assume that $W$ is isotropic at
one or more infinite places of $C$. Then $N$ and $M$ are in the
same $G$-class if and only if $\rho(N,M)=\mathrm{id}$.
\end{coro}

Next corollary follows now from (\cite{Om}, p.170) and
(\cite{scharlau}, p.363).

\begin{coro}
Let $N$ and $M$ be two $C$-lattices in the quadratic space $W$ of
dimension at least $5$ (or the skew-hermitian space $W$ of rank at
least 4) that belong to the same genus. If $C$ is any proper open
subset of $X$, then $N$ and $M$ are in the same $G$-class if and
only if $\rho(N,M)=\mathrm{id}$.
\end{coro}

Recall that the space $\lambda(X)$ of global sections of a lattice $\lambda$ is a finite
dimensional vector space over finite field $\finitum=\oink_X(X)$ (\cite{weil}, Chapter VI). Furthermore,
 for any $n$-linear map
$\tau:V^n\rightarrow W$ satisfying
$\tau\left(\lambda^n\right)\subseteq M$, where $\lambda$ is a lattice
in $V$ and $M$ is a lattice in the $K(X)$-vector space $W$, there exists an induced map
$\tilde(\tau):\lambda(X)^n\rightarrow M(X)$ that is $n$-linear over $\finitum$. In particular:
\begin{enumerate}
\item If $\lambda$ is an order, then $\lambda(X)$ is a $\finitum$-algebra.
\item If $\lambda$ is an integral quadratic lattice, then $\lambda(X)$ is, naturally, a quadratic space over
$\finitum$.
\end{enumerate}
This observation is used throughout next section.

\paragraph{Example: Maximal orders in $\matrici_2(K)$.}
If $n=2$ every split maximal order has the form
$$\Da_B=\bimatrix{\oink_X}{\mathfrak{L}^B}{\mathfrak{L}^{-B}}{\oink_X},$$
where $B$ is a divisor of $X$ defined over $\finitum$. If $B$ is a
principal divisor, the ring of global sections $\Da_B(X)$ is
isomorphic to the matrix algebra $\matrici_2(\finitum)$. If $B$ is
not principal, then $\mathfrak{L}^B$ and $\mathfrak{L}^{-B}$
cannot have a global section simultaneously. In fact, if
$\mathrm{div}(f)+B\geq0$ and $\mathrm{div}(g)-B\geq0$ then
$B=\mathrm{div}(g)=\mathrm{div}(f^{-1})$. We conclude that either
$\Da_B(X)\cong \finitum\times\finitum$ or $\Da_B(X)\cong
(\finitum\times\finitum)\oplus V$, where $V$ is an ideal of
nilpotency degree $2$. Note that the dimension of $V$ tends to
$\infty$ with the degree of $B$ by Riemann Roch's Theorem, whence
we conclude that there exist infinite many conjugacy classes of
maximal $X$-orders in $\alge$. In fact, we can give a more precise
result:

\begin{prop}
The maximal orders $\Da_B$ and $\Da_D$ defined in the previous example
 are conjugate if and only if $B$ is linearly equivalent to either $D$ or $-D$.
\end{prop}

\begin{proof}
If $B$ is principal, and if $\Da_B$ is conjugate to $\Da_D$, we
must have $\Da_D(X)\cong\matrici_2(\finitum)$, and therefore $D$
is principal. We can assume therefore that neither $B$ nor $D$ is
principal. Replacing $B$ or $D$ by $-B$ or $-D$ if needed, we may
assume $B,D\geq0$. Let $U$ be a global matrix such that
$\Da_B=U\Da_DU^{-1}$. From the explicit description of $\Da_B(X)$
given earlier, we conclude that the $\finitum$-vector spaces
$\mathfrak{L}^B(X)$ and $\mathfrak{L}^D(X)$ have the same
dimension. Furthermore, if $W_B$ and $W_D$ denote the $K$-vector
spaces spanned by $\Da_B(X)$ and $\Da_D(X)$ respectively, then
$W_B=UW_DU^{-1}$.  There are two cases two be considered:
\begin{enumerate}
\item If $\mathfrak{L}^B(X)\neq\{0\}$, then
$W_B=W_D=KE_{1,1}\oplus KE_{2,2}\oplus KE_{1,2}$. \item If
$\mathfrak{L}^B(X)=\{0\}$, then $W_B=W_D=KE_{1,1}\oplus KE_{2,2}$.
\end{enumerate}
 In the first case, we conclude that $U$ has the form
\scriptsize $\bimatrix ab0c$ \normalsize. In particular we must have
$$\mathfrak{L}^{-B}E_{2,1}=E_{2,2}\Da_BE_{1,1}=
E_{2,2}(U\Da_DU^{-1})E_{1,1}=a^{-1}c\mathfrak{L}^{-D}E_{2,1}.$$ We
conclude that $B=D+\mathrm{div}(ac^{-1})$, and therefore $B$ and
$D$ are linearly equivalent. In the second case $U$ has either the form
\scriptsize $\bimatrix a00c$ \normalsize, which is similar to the previous case,
or the form
\scriptsize $\bimatrix 0ac0$ \normalsize, so that $B=-D+\mathrm{div}(ac^{-1})$,
and $B$ is linearly equivalent to $-D$.
\end{proof}

\section{Representation fields}\label{s4}

Let $\Lambda$ be an $X$-lattice in a $K(X)$-vector space $V$ as
before, and let $M$ be an $X$-lattice in a subspace $W\subseteq V$. We allow
the case $V=W$. Assume $M\subseteq\Lambda$ in all that follows.
Assume that $G$ is a semi-simple linear algebraic sub-group of
$\mathrm{GL}(V)$ satisfying the conditions \textbf{SN} and
\textbf{RU}. Following the notations in \cite{spinor} we call an
element $u\in G_\ad$ a generator for $\Lambda|M$ if $M\subseteq
u\Lambda$. Local generators are defined analogously. Note that
$u\in G_\ad$ is a generator if and only if $u_\wp$ is a local
generator for every place $\wp$. As usual we say that $M$ is
$G$-represented by an $X$-lattice $N$ in $V$, or that $N$
$G$-represents $M$, if $N$ contains a lattice in the the $G$-orbit
of $M$. We say that a set $\Psi$ of lattices $G$-represents $M$ if
some element of $\Psi$ does. In this setting, we have the
following proposition, whose proof is transliteration of the one in
the number field case \cite{spinor}, and therefore is omited.

\begin{prop} In the above notations, let $\Lambda'$ be an
$X$-lattice in the $G$-genus of $\Lambda$.
 The lattice $M$ is
$G$-represented by the spinor genus of $\Lambda'$ if and only if
there exists a generator $u$ for $\Lambda|M$ such that
$\rho(\Lambda,\Lambda')=[\Theta(u),\Sigma_\Lambda/K(X)]$.\qed
\end{prop}

We denote by $H(\Lambda|M)$ the pre-image in $J_X$ of the set of
spinor norms $\Theta(u)$ of all generators $u$ for $\Lambda|M$. If
$K(X)^*H(\Lambda|M)$ is a group, the corresponding class field is
called the representation field $F(\Lambda|M)$ for $\Lambda|M$.
The spinor genus of a lattice $\Lambda'$ represents $M$ if and
only if $\rho(\Lambda,\Lambda')$ is trivial on $F(\Lambda|M)$.
The proof of the following fact is also completely analogous
to the number field case (\cite{Hsia98} and
\cite{repsh}):
\begin{prop} The
representation field always exist for lattices in quadratic or
quaternionic skew-hermitian spaces.\qed\end{prop}
 The corresponding result for orders in quaternion algebras follows from
 the case of quadratic forms just as in the number field case, but it cannot be
 extended to algebras of higher dimension \cite{abelianos}.

It must be kept in mind, however, that the classification of the
lattices in a genus into spinor genera is a much coarser invariant
than in the number field case, as the number of classes in a genus is
usually infinite. The following example illustrate this:

\paragraph{Example B (continued).}Let $X$ and $\Lambda$ be as in \S4,
but assume $n=4$ and the field of definition of $X$ is
$\finitum=\finitum_q$ with $q=4t+3$. In particular, the space of
global sections of $L(B)$ is
$$\Lambda(X)=\oink_X(X)v_1+\oink_X(X)v_2+\mathfrak{L}^B(X)v_3+\mathfrak{L}^{-B}(X)v_4.$$
By Riemann-Rochs Theorem, the $\finitum$-dimension of $\Lambda(X)$
tends to infinity with the degree of $B$. In particular there
exists infinitely many classes of such lattices, while they belong
to the same spinor genus as long as $\mathrm{deg}(B)$ is even. We claim that none of
these lattices represents $M=\oink_Xv_3+\oink_Xv_4$ when
$B$ is not principal. In fact, if $B$ fails to be principal, then either
$\mathfrak{L}^B(X)$ or $\mathfrak{L}^{-B}(X)$ has dimension $0$, whence the space of global sections
$\Lambda(X)=[\oink_X(X)v_1+\oink_X(X)v_2]\perp U$ where $U=\mathfrak{L}^B(X)v_3+\mathfrak{L}^{-B}(X)v_4$
 is the radical, and $Z=\oink_X(X)v_1+\oink_X(X)v_2$ is anisotropic by the choice of $q$,
  while $M(X)$ is a hyperbolic plane.
We note however that the theory of representation by spinor genera tell us that half of the spinor genera
in the genus of $\Lambda$ must represent $M$. It is not hard to see that the image of the spinor norm in this case is
$H(\Lambda)=J_X^n\oink_\ad^*$. It follows that there are more than one spinor genus representing $M$ whenever the
torsion subgroup of the Picard group of $X$ has even order. In this case there must exist classes in the genus of
$\Lambda$ that are not in the class of any of the lattices $L(B)$.

Recall that an order of maximal rank in a central simple algebra $\alge_{K(V)}$ is said to be split if it represents
the n-fold cartesian product $\oink_X\times\cdots\times\oink_X$.
A maximal $X$-order $\Da$ in $\matrici_n(K)$ is split if and only if
the corresponding vector bundle is a direct product of one
dimensional vector bundles, i.e., it
corresponds to an $X$-lattice of the type
$$\Lambda=\mathfrak{L}^{B_1}\times\cdots\times\mathfrak{L}^{B_n}.$$
 Note that if the order of diagonal matrices $\bigoplus_i\oink_XE_{i,i}$
is contained in $\Da$, every diagonal matric unit $E_{i,i}$ is a global section
of $\Da$, an therefore
$\Da=\sum_{i,j}J_{i,j}E_{i,j}$ for some invertible bundle
$J_{i,j}\subseteq K$, and the same holds for the rings of global
sections. A simple computation shows
$J_{i,j}\cong\mathfrak{L}^{B_i-B_j}$. An argument similar to that in the previous example
 can be used to prove next result:

\begin{prop}\label{maingt}
If $N$ is the total number of spinor genera of maximal $X$-orders
in the matrix algebra $\matrici_n(K)$, where $K$ is the field of
functions on a smooth projective curve $X$ over a finite field,
then at least $\frac{N}n-1$ spinor genera contain non-split
$X$-orders.
\end{prop}

\begin{proof}
 One particular example of split order is the maximal
order $\Da_B$ corresponding to the $X$-lattice
$\mathfrak{L}^B\times\oink_X\times\cdots\times\oink_X$.
We claim that \begin{enumerate}\item
Every spinor genera of maximal $X$-orders contains the order
$\Da_B$ for some divisor $B$. \item The maximal orders $\Da_B$ and
$\Da_D$ are in the same spinor genus if and only if $B$ and $D+nC$
are linearly equivalent for some divisor $C$.
\end{enumerate}
The stabilizer of the local maximal $X$-order $\Da_\wp$ is
$K_\wp^*\Da_\wp^*$ and its set of norms is
$K_\wp^{*n}\oink_\wp^*$. We conclude that
$H(\Da)=J_K^n\oink_{\mathbb{A}}^*$. It follows that the class
group $J_K/K^*H(\Da)$ is isomorphic to the divisor group of $X$
modulo $n$-powers. Observe that $\Da_D=u\Da_B u^{-1}$
where $u=\mathrm{diag}(b,1,\dots,1)$, and the idele $b=n(u)\in J_K$
satisfies $\mathrm{div}(b)=D-B$. It follows that $\Da_D$ and
$\Da_B$ are in the same spinor genus if and only if $D-B$ is $0$
modulo $n$-powers in the divisor group of $X$. Furthermore, any spinor genera
can be obtained in this way for an apropiate choice of $b$ in $J_X/K(X)^*H(\Da)$.
In particular, every spinor genus contains a split order.

It follows from the previous argumennt that the class modulo $n$ of
the divisor $B$ depends only on the spinor genera of the maximal
order $\Da_B$. In particular, the degree of $B$ is well defined
for a particular spinor genus as an element of $\enteri/n\enteri$.
We use this in all that follows.

Let $\mathbb{L}$ be the only field extension of the finite field
$\finitum$ of degree $n$. We claim that,
 if $\mathbb{L}$ embeds into $\Da(X)$ for
a split order $\Da$, then $\Da\cong\matrici_n(\oink_X)$.
In fact, let
$\Lambda=\mathfrak{L}^{B_1}\times\cdots\times\mathfrak{L}^{B_n}$
be the lattice corresponding to $\Da$. Then
$\Da=(\mathfrak{L}^{B_i-B_j})_{i,j}$. We define an order in the
group of divisor classes by $D\preceq C$ if $D\leq
C+\mathrm{div}(f)$ for some $f\in K$. Note that
$\mathfrak{L}^{C-D}$ has a non-trivial global section if and only
if $D\preceq C$. We assume that the $B_i$'s have been re-arranged
in a way that $B_n$ is minimal with respect to this order, and
$B_{r+1},\dots,B_n$ are all the divisors that are linearly
equivalent to $B_n$. Then any global section of $\Da$ has the form
$$\bimatrix AB0C$$ where $A$ is an $r$-times-$r$ block. It follows
that $K\mathbb{L}$ has a representation of dimension $r<n$ over
$K$, and therefore $r=0$.

Now, the proposition follows if we prove that for any divisor $B$
 such that $\mathrm{deg} B\equiv0\
(\mathrm{mod}\ n)$, there exists a maximal $X$-order $\Da$ in the
same spinor genus as $\Da_B$ for which there is an embedding
$\mathbb{L}\hookrightarrow\Da(X)$, since we know that $\Da$ cannot be
a split order unless $B$ is principal.

To prove this we let $L=K\mathbb{L}=\mathbb{L}\otimes_{\finitum}K$, and let
$\Ha=\mathbb{L}\otimes_{\finitum}\oink_X$ be the only maximal
order in the $K$-algebra $L$. Note that if $Y$ is the projective
curve over $\mathbb{L}$ defined by the same equations defining $X$
over $\finitum$, and $\phi:Y\rightarrow X$ is the natural morphism
of schemes, then $\Ha$ is the push-forward to $X$ of the structure
sheaf on $Y$. In particular $\Ha(X)=\mathbb{L}$.

Consider the natural embedding
$$\phi:\Ha=\mathbb{L}\otimes_{\finitum}\oink_X\hookrightarrow
\matrici_2(\finitum)\otimes_{\finitum}\oink_X=\Da_0$$ induced by
an arbitrary embedding
$\mathbb{L}\hookrightarrow\matrici_2(\finitum)$. Then
 the order $\Ha'=\phi(\Ha)$ is contained in some
maximal order in the spinor genera of $\Da_B$ if and only if we
can write $\Da_B=u\Da_0u^{-1}$ where the image of the reduced norm
$n(u)$ in the quotient $J_K/K^*H(\Da)$ coincide with the image of a generator.
 Note that if
we identify $L$ with the sub-algebra of $\alge$ spanned by $\Ha'$,
the group of invertible elements $L_{\mathbb{A}}^*$ (all of which
are generators for $\Da|\Ha'$) is isomorphic to the group of
ideles $J_L$ of $L$, and the reduced norm $n:J_L\rightarrow J_K$
is just the field norm $n_{L/K}$. It follows from Theorem 7 in
chapter XIII of \cite{weil}, that $H_L=K^*n_{L/K}(J_L)$ is the
kernel of the Artin map $t\mapsto [t,L/K]$ on ideles. In
particular $H_L$ has index at most n in $J_K$, and we can check
that the divisor of every idele in $H_L$ has degree in $n\enteri$
by computing the degrees of the generators. We conclude that $H_L$
is the group of all ideles whose divisors have degrees in
$n\enteri$, whence the result follows.\end{proof}

\begin{rmk}
Since for every curve $\mathrm{Pic}(X)\cong \enteri\times T$ where
$T\cong\mathrm{Pic}^0(X)$ is a finite group (\cite{weil}, \S IV.4,
Theorem 7), we conclude that
$J_K/K^*H(\Da)\cong(\enteri/n\enteri)\times (T/nT)$. In
particular, the bound in the proposition is $|T/nT|-1$.
\end{rmk}

\begin{rmk}Assume for simplicity that $\mathbb{K}$ has odd
characteristic and $n=2$. Let $B=\mathrm{div}(b)$ be a divisor of
even degree, with $b\in J_K$. An order $\Da$ in the same spinor
genera as $\Da_B$, representing the maximal order of $L$, is given
as follows: Tchebotarev Density Theorem (\cite{rosen}, Thm. 9.13A)
implies the existence of a place $\wp\in X$, such that any idele
$j$, where $j_\wp$ is a uniformizer of $K_\wp$, and
$j_{\mathfrak{q}}=1$ if $\mathfrak{q}\neq\wp$, satisfies
$bj^{-1}\in K^*H(\Da)$. Note that $\wp$ has even degree, whence
the field $\mathbb{L}$ embeds into $K_\wp$. We may assume
$\mathbb{L}=\finitum(u)$, where $u$ is a root of $x^2=\delta$ for
some $\delta \in\finitum$. Then $\mathbb{L}K$ embeds into
$\matrici_2(K)$ by sending $u$ to the matrix \scriptsize $\bimatrix
01\delta0$\normalsize. Let $P$ be the divisor corresponding to
$\wp$. Then we may choose $\Da=b\Da_0 b^{-1}$ where
$b_{\mathfrak{q}}$ is the identity matrix for
$\mathfrak{q}\neq\wp$ and
$$b_\wp=A\bimatrix {j_\wp}001 A^{-1}, \textnormal{ where }A\bimatrix u00{-u}
A^{-1}=\bimatrix 01\delta0.$$ Note that $\det b=j$.
\end{rmk}

\begin{ex}
When $X=\mathbb{P}^1$ is the projective plane, the lower bound
given by this result is $0$, so in principle there could be no
non-split $X$-orders in $\matrici_n(K)$. In fact, the
Birkhoff-Grothendieck Theorem\footnote{ A proof of this result for
arbitrary fields is given in \cite{burban} (Theorem 2.1)} implies
that every such $X$-lattice is a sum of $X$-lattices of rank $1$.
It follows that non-split orders fail to exist and the bound is
sharp in this case.
\end{ex}

\begin{ex}
Consider the plane curve $X$ of genus $1$ with projective equation
$y^2z-x(x^2-z^2)=0$ over a finite field of odd characteristic.
Then $\mathrm{div}(x)=2(P_0-P_\infty)$, while there is no element
in $K=K(X)$ whose divisor is $P_0-P_\infty$, or such element would
generate the field $K$. We conclude that $\mathrm{Pic}^0(X)$ has
an element of order $2$, and therefore its order is even. We
conclude the existence of at least one class of non-split maximal
orders in $\matrici_2(K)$, or equivalently, a non-split vector
bundle defined over $\finitum$.
\end{ex}

\end{document}